\documentclass[10pt]{amsart}
\usepackage{a4}
\usepackage{amsmath,amssymb,amsfonts,amsthm}

\usepackage{color}

\def\R{\mathbb{R}}
\def\Q{\mathbb{Q}}
\def\p{\phi}
\def\po{\phi^\circ}
\def\Div{\textup{div}}
\def\dist{\textup{dist}}
\def\spt{\textup{Supp}\,}
\def\e{\varepsilon}
\renewcommand{\H}{\mathcal{H}}
\newcommand{\beq}{\begin{equation}}
\newcommand{\eeq}{\end{equation}}
\newcommand{\pa}{\partial}
\newcommand{\noop}[1]{}
\theoremstyle{plain}
\newtheorem{theorem}{Theorem}[section]
\newtheorem{proposition}[theorem]{Proposition}
\newtheorem{corollary}[theorem]{Corollary}
\newtheorem{lemma}[theorem]{Lemma}
\newtheorem{definition}[theorem]{Definition}
\theoremstyle{remark}
\newtheorem{remark}[theorem]{Remark}

\newcommand{\dd}{d}
\newcommand{\N}{\mathbb N}
\numberwithin{equation}{section}
 \title{Existence and uniqueness for a crystalline mean curvature flow}
\author{Antonin Chambolle \and Massimiliano Morini \and Marcello Ponsiglione}

\begin{document}

\maketitle

\begin{abstract}
\small{An existence and uniqueness result, up to  fattening,  for a class of crystalline mean curvature flows
with  natural mobility is proved. The results are valid in any dimension and for arbitrary, possibly unbounded, initial closed sets.
The comparison principle is obtained by means of a  suitable weak formulation of the flow, while the existence  of a global-in-time solution follows via a  minimizing movements approach.
 
}
\end{abstract}

\bibliographystyle{plain}

\section{Introduction}

In this paper we deal with the anisotropic mean curvature  motion; that is,  flows of sets $t\mapsto E(t)$ (formally) governed by the law 
\begin{equation}\label{oee} 
V(x,t) = -m(\nu^{E(t)})\kappa^{E(t)}_{\p}(x),
\end{equation}
where  $V(x,t)$ stands for the (outer) normal velocity of the boundary $\pa E(t)$ at $x$, $\p$ is a given norm on $\R^N$ representing the {\it surface tension}, $\kappa^{E(t)}_{\phi}$ is the {\em anisotropic mean curvature} of $\pa E(t)$ associated with the anisotropy $\p$, and $m$ is a positive {\em mobility} which depends on the outer unit normal $\nu^{E(t)}$ to $\pa E(t)$. 
Such an evolution law may be regarded as the gradient flow (with respect to a suitable formal Riemannian structure) of the anisotropic perimeter functional
\beq\label{pfi}
P_\p(E)=\int_{\pa E}\p(\nu^E)\,d\H^{N-1},
\eeq
 the anisotropic curvature $\kappa^{E}_{\p}$ of  $\pa E$ being nothing but the first variation of \eqref{pfi} at $E$.  When $\p$ is differentiable in $\R^{N}\setminus\{0\}$, then $\kappa^{E}_{\p}$ is given by
\begin{equation}\label{kappaphi}
\kappa^{E}_{\p}=\Div\left(\nabla \p(\nu^E)\right)\,.
\end{equation}
However, we are particularly interested in the case when $\phi$ is not
differentiable, for instance the {\em crystalline} case, when the unit ball $B_{\p}:=\{\p\leq 1\}$, known as the \textit{Frank diagram}, is a convex polytope.  In the latter case, we will only consider the natural mobility given by $m=\p$. With this choice,  \eqref{oee} has the interesting property that the flow starting from a  \textit{Wulff shape}, that is, a level
set of the polar $\po$ of $\p$, consists in a one-parameter family of shrinking Wulff shapes that extinguish in finite time. We recall that  Wulff shapes are the only solutions to the  isoperimetric problem associated with $P_\p$ (see \cite{FonsecaMuller}). 

The law \eqref{oee} is used to describe several phenomena in Materials Science and Crystal Growth, see for instance \cite{Taylor78, Gurtin93}.
From the mathematical point of view, the geometric motion is well defined
in a classical sense in the smooth case, that is, when $\p$ is 
at least $C^{3,\alpha}$ (as well as the initial surface, and
except at the origin) and ``elliptic'' (which means for instance
that $\p^2$ is strongly convex)~\cite{ATW}.
Of course, the classical mean curvature flow falls within this class and corresponds to the choice $\p=\text{Euclidean norm}$. In the smooth case, the main mathematical difficulties are related to the fact that singularities (like pinching) may form in finite time (see for instance \cite{Grayson89}) in dimensions $N\geq 3$.  Thus, 
the strong formulation of  \eqref{oee}, which requires smoothness of the evolving sets, is well defined only for short times and one needs a weaker notion of solution that can handle the presence of singularities in order to define the flow for all positive times. When $\p$ is smooth, this task has been already accomplished and different approaches have been proposed in the literature, starting from the pioneering work 
 by  Brakke~\cite{Br}, who suggested
 a weak  formulation  of the motion by mean curvature yielding deep  regularity results but  lacking  uniqueness. These uniqueness issues have been subsequently overcome via  the so-called {\em level set approach}~\cite{OS,EvansSpruckI,CGG}.
In particular, the case of~\eqref{oee} for $m,\p$ of class $C^2$ is covered
by~\cite{CGG}.
The main idea is to  represent the initial set as   the zero sublevel set of a function $u_0$ and then to let all these level sets evolve  according to the same geometric law (which makes sense thanks to the fact that the evolutions
which we consider preserve inclusion).
This procedure defines  a time-dependent function $u(x,t)$ and transforms the geometric equation into a (degenerate) parabolic equation for $u$,  which is shown to admit a  {\em unique viscosity solution} with the prescribed initial datum $u_0$. The evolution of the zero sublevel set of such a solution defines a {\em generalized motion} (see also~\cite{BaSoSou}), which exists for all times and agrees with the classical one for short times, before the appearance of singularities (see \cite{EvansSpruckII}). Such a motion satisfies a comparison principle and is unique whenever the level sets of $u$ have zero Lebesgue measure, i.e., whenever the so-called fattening phenomenon does not occurs. Fattening may in fact appear even for a smooth initial datum $E^0$ (see \cite{AngIlCh}), but its occurrence  is in some sense very ``rare'':  for instance, it is easy to understand that  almost all the sublevels sets of the signed distance function from any given set $E^0$ will not generate any fattening. 
 
 A third approach is represented by  the minimizing movements scheme devised by
Almgren, Taylor and Wang~\cite{ATW} and Luckhaus and Sturzenhecker~\cite{LS}. 
It consists in constructing a sequence of discrete-in-time evolutions by iteratively solving  suitable incremental minimum problems.
 Any limit  of these  evolutions as the time step vanishes defines a   motion, which exists for all positive times (and is shown to be H\"older-continuous
in time for the $L^1$ norm).
The connections between the generalized level set motion and Brakke solutions has been investigated in \cite{Ilmanen94}.
 A simple proof of convergence of the Almgrem-Taylor-Wang (ATW) to
the generalized motion is shown in \cite{ChambolleNovaga}, while a
consistency result was already shown in~\cite{ATW}. See
also~\cite{EGK} for a similar convergence proof in a more general
setting (allowing for unbounded surfaces, as in the present paper),
and~\cite{LauxOtto15} for new proofs and a generalization to partitions.
Roughly speaking, it turns out that whenever fattening does not occur, the generalized level set motion coincides with the ATW flow and is also a solution in the sense of Brakke.

Let us now  consider a crystalline anisotropy. This case is more difficult, due to the lack of smoothness in the involved differential operators. Indeed, the crystalline normal $\nabla \p(\nu^E)$ is not uniquely defined for some directions and one needs to look at  suitable selections of the (multivalued) subdifferential map, that is,   vector fields  $z:\pa E\to \R^N$, such that $z(x)\in \pa\p(\nu^E(x))$ for a.e. $x$. If there exists an admissible field $z$ with tangential divergence $\Div_\tau z$ in $L^2(\pa E)$, then the crystalline curvature is given by the tangential divergence of $z$, where $\Div_\tau z$ has minimal $L^2$-norm  among all admissible fields (see \cite{BeNoPa,GigaGigaPozar}). In particular, the crystalline curvature has a nonlocal character.  

Showing (even local-in-time) existence and uniqueness for crystalline mean curvature flows is somewhat harder and still largely open. Only in  dimension 2, the problem has been settled  by developing a crystalline version of the viscosity approach for the level-set equation, see \cite{GigaGiga01}. If the initial
set is itself an appropriate planar crystal, the evolution equation boils down
to a system of ODEs which has been studied in many former works,
see in particular~\cite{AlmTay95, AngGu, AngGu89, GigaGiga98, GiGu96},
while existence and uniqueness
of strong solutions for initial ``regular'' (in an appropriate
sense) sets was shown recently in~\cite{ChaNov-Crystal15}.
One advantage of the level-set approach of~\cite{GigaGiga01} is the
ability to address much more general equations where the speed
depends on the crystalline curvature and the normal in an non-linear way.

In dimensions $N\geq 3$, the only general available notion of global-in-time solution we are aware of is the minimizing movements motion provided by the ATW scheme; however, no general comparison results have been established so far.
In fact, the  higher-dimensional uniqueness results we know of deal with special classes of initial data (for instance convex initial data as in  \cite{CaCha, BelCaChaNo} or polyhedral sets as in \cite{GiGuMa98}) or with very specific anisotropies (see \cite{GigaGigaPozar} where a comparison principle valid in all dimensions has been established for the anisotropy $\p(\nu)=|\nu'|+|\nu_N|$, with $\nu_N:=\nu\cdot e_N$ and $|\nu'|$ the Euclidean norm of the orthogonal projection of $\nu$ onto $e_N^{\perp}$). However, Y.~Giga has recently announced a
very general existence and uniqueness result in the viscosity sense in dimension $N=3$.

In this paper we prove a global-in-time existence and uniqueness (up to possible fattening) result for the crystalline mean curvature flow {\em valid in all dimensions, for arbitrary (possibly unbounded) initial sets, and for general crystalline anisotropies $\p$},  but under the particular choice $m=\p$ in \eqref{oee}. We do so by providing a suitable weak formulation of the problem and then by showing that such a notion yields a comparison principle. We then implement a variant of the ATW scheme to establish an existence result. 

Le us describe our approach in more details. It is based on ideas
of~\cite{Soner93,AmbrosioSoner}. 
 In order to motivate our formulation, let us  assume for a moment that $\p$ is smooth and that 
$t\mapsto E(t)$ is a regular flow obeying \eqref{oee}. Set $d(\cdot, t):=\dist(\cdot, E(t))$, where $\dist$ denotes the distance induced by the polar norm $\po$ (see \eqref{polar} and \eqref{polardist} below). Then it is easy to see that the time partial derivative $\partial_td$ of $d$ on $\pa E(t)$ equals  $-V/\p(\nu^{E(t)})$, with $V$ denoting the outer normal velocity of the moving boundary. 
On the other hand, this quantity $V/\p(\nu^{E(t)})$ is nothing else as
the speed of the moving boundary along the \textit{Cahn-Hoffmann} normal
$\nabla\p(\nu^{E(t)})$, see~\cite{Gurtin93,BeNoPa}.
Thus, \eqref{oee} may be rewritten as 
$$
\partial_td=\kappa_\p^{E(t)}= \Div (\nabla\p(\nabla d)) \qquad\text{on $\pa E(t)=\pa \{d(\cdot, t)=0\}$}.
$$
(Here and throughout the paper $\nabla$ stands for the spatial gradient.) On the other hand, if we look at a  positive $s$-level set of $d$, the  (weighted) normal velocity of $x\in \{\dist(\cdot, t)=s\}$ equals the normal velocity of its projection $y$ on   $\pa E(t)$, which is given by the anisotropic curvature $\kappa_\p^{E(t)}(y)$ of $\pa E(t)$ at $y$. Since (as long as the surfaces are
smooth)
$$
\kappa_\p^{\{d(\cdot, t)=s\}}(x)= \Div (\nabla\p(\nabla d))(x,t)\leq \kappa_\p^{E(t)}(y),
$$
 we deduce that
\begin{equation}\label{distform1}
\partial_td\geq \Div (\nabla\p(\nabla d))\qquad \text{in $\{d>0\}$}
\end{equation}
as long as $E(\cdot)$ is nonempty. In words, the positive level sets of the distance function shrink with a velocity which is higher than that given by the anisotropic curvature, and thus they may be regarded as  super-flows or {\em supersolutions} of the geometric motion.
Analogously, setting $d^c(\cdot, t):=\dist(\cdot, E^c(t))$, where $E^c$ stands for the  complement of $E$, we have
\begin{equation}\label{distform2}
\partial_td^c\geq \Div (\nabla\p(\nabla d^c))\qquad \text{ in $\{d^c>0\}$}
\end{equation}
as long as $E^c(\cdot)$ is nonempty.
We may conclude that a smooth flow  
$t\mapsto E(t)$ of sets solves \eqref{oee} if and only if  \eqref{distform1} and \eqref{distform2} are satisfied. 

As already remarked before, when $\p$ is crystalline $\nabla\p(\nabla d)$ may not be defined and must be replaced in general by a suitable selection of the subdifferential map, that is, by a vector-field $z\in L^{\infty}(\{d>0\}; \R^N)$ such that $z(x)\in \pa\p(\nabla d(x))$ for a.e. $x$, where $\pa\p$ denotes the subdifferential of $\p$. Any such  $z$ will be called {\em admissible for $d$}.

The above discussion motivates the following weak formulation of the crystalline flow: we will say that a one-parameter family  $t\mapsto E(t)$ of closed sets, satisfying suitable continuity properties (see Definition~\ref{Defsol}  below) is a {\em weak supersolution} of \eqref{oee} with initial datum $E^0$ if  $E(0)\subseteq E^0$ and there exists a vector-field $z$, admissible for $d$,  such that   \eqref{distform1}  hold in the sense of distributions, with $\nabla\p(\nabla d)$  replaced by $z$.   We will say instead that $t\mapsto E(t)$ is a weak-subsolution of \eqref{oee} if 
 $E(0)\supseteq E^0$ and  $t\mapsto (\mathring{E}(t))^c$  is  weak supersolution. Finally, we will say that $t\mapsto E(t)$ is a {\em weak solution} if it is both a weak sub- and a supersolution (with initial datum $E^0$).
Mostly for technical reasons, we will require in addition
that the positive part of $\Div z$ is bounded in $\{d\geq \delta\}$ for all $\delta>0$.

Let us notice that this formulation of the curvature flow in terms of the distance function has been already exploited for the standard mean curvature motion and its regular anisotropic variants. In fact, it is  close in spirit to the distance formulation proposed and studied in \cite{Soner93}, although it  is somewhat stronger as it require the differential inequalities to hold in a distributional sense,  rather than in the viscosity sense considered in \cite{Soner93}.
In this respect, our formulation is reminiscent 
of the approach developed in  \cite{CaCha}. 

We now describe the plan of the paper. In Section~\ref{sec:wf}, after recalling some preliminaries definitions and introducing the main notation, we give the precise weak formulation of the sub- and supersolutions to the anisotropic mean curvature flow.  In Section~\ref{sec:comp} we establish a comparison principle between sub- and supersolutions, which by standard arguments yields the uniqueness of the crystalline flow whenever fattening does not occur. 
We remark that the distributional formulation described above allows for a proof of the comparison,   which is closer in spirit to the uniqueness proofs for standard parabolic equations. In particular, our argument is more elementary than the typical   ``viscosity''  proof that is based on delicate regularization procedures and  fine differentiability properties of semiconvex functions. 
In Section~\ref{sec:ATW} we provide an existence results for the the weak formulation of the crystalline flow, which is based on the reformulation of the minimizing movements scheme of Almgren-Taylor-Wang / Luckhaus-Sturzenhecker introduced
 in~\cite{Chambolle,CaCha}. 
Such a variant can be considered as a combination of the ideas of \cite{ATW} and the threshold dynamics algorithm studied in \cite{Evans93}, and has several advantages: for instance, it makes it easier to establish a comparison principle for the discrete-in-time evolutions and it works equally well for bounded and unbounded sets (as already exploited in~\cite{EGK}). In the main theorem of the section, we establish the convergence of the minimizing movements scheme to 
a weak solution, whenever no fattening occurs. 

We conclude this introduction by commenting on the restriction $m=\p$ in \eqref{oee}. Although such a mobility is rather natural (for instance it forces  Wulff shapes to evolve in a self-similar way),  it is not the most general case and different mobilities could be considered as physically interesting. However, at the moment, in the crystalline case we are able to provide the right convergence estimates for the minimizing movements scheme only under this assumption;  the main technical reason is  related to the fact that if $\dist$ is the distance induced by the polar norm $\po$, then  the crystalline curvatures of the positive level sets of  $\dist(\cdot, E)$ are bounded above (this can be easily understood since in this case the sublevel sets of    $\dist(\cdot, E)$ admit an inner tangent Wulff shape at all points of the boundary). 
Nevertheless, we remark that in the case of a smooth elliptic anisotropy, all our results and methods would work with {\em any} mobility $m$, thus showing that the viscosity solutions already studied in \cite{EvansSpruckI, CGG, Soner93}  satisfy in fact a stronger (distributional) formulation.
The extension of our results to more general mobilities in the crystalline case will be the subject of future investigations. 



\section*{Acknowledgements}
Part of this research was done in the Institut Henri Poincar\'e
in Paris, where M.~Morini and M.~Ponsiglione were hosted for a month
in 2015 thanks to the ``Research in Paris'' programme of this mathematical
institute. The authors are very grateful for this support.
In addition, A.~Chambolle was partially supported by the ANR, programs
ANR-12-BS01-0014-01 ``GEOMETRYA'' and ANR-12-BS01-0008-01  ``HJnet''.

\section{A weak formulation of the crystalline mean curvature flow}\label{sec:wf}
In this section we introduce a suitable weak formulation  of the crystalline mean curvature flow. Such a notion of solution resembles the formulation  due to \cite{Soner93}. However, here we will not consider the viscosity setting of \cite{Soner93} and we will rather be concerned with distributional solutions (which appear for instance in \cite{CaCha}).
\subsection{Preliminaries}
In this subsection we  introduce the main objects and notation used throughout the paper.

Let $\p$ denote a fixed norm on $\R^N$, that is, a convex, even and $1$-homogeneous real-valued function, which will play the role of the anisotropic interfacial energy density.
In the terminology of crystal growth this is also called {\em  surface tension}. Note that we do  not assume any further regularity on $\p$ and in fact the main case of interest is when $\p$ is {\em crystalline}, that is, when the associated unit ball is a convex polytope.  
The interfacial energy is then given by
$$
P_\p(E):=\sup\biggl\{\int_E\Div \zeta\, dx: \zeta\in C^1_c(\R^N; \R^N),\, \po(\zeta)\leq 1\biggr\}\,,
$$
where we recall that the {\em polar norm} $\po$ is defined as
\begin{equation}\label{polar}
\po(\xi):=
\sup_{\p(\eta)\le 1}\eta\cdot\xi\,.
\end{equation}
It can be checked that $P_\p(E)$ is finite if and only if $E$ is a set of finite perimeter and, in this case,
$$
P_\p(E)=\int_{\pa^*E}\p(\nu^E)\,d\H^{N-1}\,, 
$$
where $\pa^*E$ denotes the so-called reduced boundary of $E$ (see for instance \cite{AmFuPa:00}).
More generally, given a function $u\in BV_{loc}(\R^N)$   we may consider the {\em anisotropic total variation maesure} of $u$, which  on the open (bounded if $u\not \in BV(\R^N)$) subsets $\Omega\subset\R^N$ is defined
as
$$
\p(Du)(\Omega):=\sup\biggl\{\int_\Omega u\, \Div \zeta\, dx: \zeta\in C^1_c(\Omega; \R^N),\, \po(\zeta)\leq 1\biggr\}.
$$ 
  Because of the homogeneity of $\p$ it turns out that 
$\p(Du)$ coincides with the nonnegative Radon measure in $\R^N$ given by $\nabla u\, dx+ \p\left(\frac{D^s u}{|D^s u|}\right)|D^s u|$, where
$\nabla u$ stands for the absolutely continuous part of $Du$ and $\frac{D^s u}{|D^s u|}$ denotes the Radon-Nykodim derivative of the singular part $D^su$ of $Du$ with respect to its (isotropic) total variation $|D^su|$, see \cite{AmFuPa:00}.

Among the important properties of $\p$ and $\po$ let us mention the fact
that $\partial\p(0)=\{\xi:\po(\xi)\le 1\}$ while $\partial\po(0)=\{\xi:\p(\xi)\le 1\}$.
Moreover, for $\eta\neq 0$
\beq\label{subp}
\partial\p(\eta) =\{\xi:\po(\xi)\le1\textup{ and }\xi\cdot \eta=\p(\eta)\}= \{\xi:\po(\xi)=1\textup{ and }\xi\cdot \eta=\p(\eta)\}
\eeq
(and the symmetric statement for $\po$).
An easy  consequence of the above characterization is that if $\eta\in \partial \po(x)$
and $x\neq 0$, then $x/\po(x)\in \partial \p(\eta)$.

The set 
$$
W(0,1):=\{y:\po(y)\le 1\}
$$ 
is called the {\em Wulff shape} associated with $\p$.  More 
generally, for  $x\in\R^N$ and $R > 0$, we will denote by
$$
W(x,R):=\{y:\po(y-x)\le R\}
$$ 
the {\em Wulff shape of radius $R$ and center $x$}.
In the Finsler metric framework associated with $\po$,  Wulff shapes play the same role as standard balls do in the Euclidean setting. In particular,   
it is well-known that $W(0,R)$ is the unique (up to translations) solution of the anisotropic isoperimetric problem 
$$
\min\left\{P_\p(E):\,  |E|=|W(0,R)|\right\},
$$
see for instance~\cite{FonsecaMuller}.


Given a  set $E\subseteq \R^N$, we denote by $\dist(\cdot, E)$ the  distance from $E$ induced by $\po$, that is, for any $x\in \R^N$
\begin{equation}\label{polardist}
\dist(x, E):=\inf_{y\in E}\po(x-y)
\end{equation}
if $E\neq\emptyset$ and  $\dist (x,\emptyset):= + \infty$.  
Moreover,  we denote by $\dd_E$ the signed distance from $E$ induced by $\po$, i.e., 
$$
\dd_E(x):= \dist(x,E) - \dist (x,E^c)
$$
so that $\dist(x,E)=\dd_E(x)^+$ and $\dist(x,E^c)=\dd_E(x)^-$ (here and throughout the
paper we adopt the  standard notation  $t^+:=t\lor  0$ and $t^-:=(-t)^+$).
Note that $\p(\nabla d_{E})=1$ a.e.~in $\R^N\setminus \partial E$.

We finally recall the notion of Kuratowski convergence. We say that a sequence of closed sets $E_n$ in $\R^m$ converges to a closed set $E$ in the Kuratowki sense, and we write
$$
E_n\stackrel{\mathcal K}{\longrightarrow} E,
$$
if the following conditions are satisfied:
\begin{itemize}
\item[(i)] if $x_n\in E_n$, any limit point of $\{x_n\}$ belongs to $E$;
\item[(ii)] any $x\in E$ is the limit of a sequence $\{x_n\}$, with $x_n\in E_n$.
\end{itemize}
One can easily see that $E_n\stackrel{\mathcal K}{\longrightarrow} E$ if and only if $\dist(\cdot, E_n)\to \dist(\cdot, E)$ locally uniformly in $\R^m$ (here one may consider the distance associated to any norm). In particular, by the Ascoli-Arzel\`a Theorem, any sequence of closed sets admits a  subsequence which converges in the Kuratowski sense.  
\subsection{A weak formulation of the crystalline flow}
In this subsection we introduce the  weak formulation of the crystalline flow we will deal with. We refer the reader to the introduction for the motivation  behind this definition. 
\begin{definition}\label{Defsol}
Let $E^0\subset\R^N$ be 
a closed set. 
Let $E$ be a closed set in $\R^N\times [0,+\infty)$ and
for each $t\geq 0$ denote $E(t):=\{x\in \R^N:\, (x,t)\in E\}$. We
say that $E$ is a {\em supersolution} of the curvature flow \eqref{oee} with
initial datum $E^0$ if
\begin{itemize}
\item[(a)] $E(0)\subseteq {E}^0$;
\item[(b)] 
for all $t\ge 0$ if ${E}(t)=\emptyset$, then $E(s)=\emptyset$ for all $s > t$;
\item[(c)]  
$E(s)\stackrel{\mathcal K}{\longrightarrow} E(t)$ as $s\nearrow t$ for all $t>0$ (left-continuity);

\item[(d)] setting $d(x,t):=\dist (x, E(t))$ for  $(x,t)\in \R^N\times (0,T^*)\setminus E$ and 
 $$
 T^*:=\inf\{t>0:\, E(s)=\emptyset \text{ for $s\geq t$}\}\,,
 $$
then the inequality
\begin{equation}\label{eq:supersol}
 \partial_t d \ge \Div z
\end{equation}
holds in the distributional sense in $\R^N\times (0,T^*)\setminus E$
for a suitable $z\in L^\infty(\R^N\times (0,T^*))$ such that
$z\in \partial\p(\nabla d)$~a.e., $\Div z$ is a Radon measure in $\R^N\times (0,T^*)\setminus E$, and 
$(\Div z)^+\in L^\infty(\{(x,t)\in\R^N\times (0,T^*):\, d(x,t)\geq\delta\})$ for every $\delta>0$.
\end{itemize}

We say that $A$, open set in $\R^N\times [0,+\infty)$, is
a subsolution with initial datum $E^0$ if $A^c$ is a supersolution with initial datum $(\mathring{E}^0)^c$.

Finally, we say that $E$, closed set in $\R^N\times [0,+\infty)$,  is a solution with initial datum $E^0$ if it is a supersolution and if $\mathring{E}$ is a subsolution, both with initial datum $E^0$.
\end{definition}


\begin{remark}\label{rm:measure}
Notice that the initial condition for supersolutions may be rewritten as $ \mathring E^0 \subseteq A(0)$.
In particular,  if $\partial{E}^0=\partial\mathring{E}^0$ and $E$ is a solution according to the previous definition, then $E(0)=E^0$.

\end{remark}
\begin{remark}\label{rm:visco}
If $\p$ is $C^2$, then one can check that this definition
is stronger than the definition in the viscosity sense (see in 
particular~\cite{Soner93,BaSoSou}).
\end{remark}
We start by observing some useful continuity properties of the map $d$ introduced in the previous definition. 

\begin{lemma}\label{lem:rightcont} Let $E$ be a supersolution. 
Then, for each $t\in  [0,T^*)$, 
$d(\cdot,s)$ converges locally uniformly in $\{x: d(x, t)>0\}$ as $s\searrow t$
to for some function $ d^r$ with $d^r \geq d(\cdot, t)$ in $\{x: d(x, t)>0\}$.
\end{lemma}
\begin{remark}
Observe that by condition (c) in the definition (which is mostly
technical and forbids artificial constructions such as a supersolution
which jumps to $E(t)=\R^N$ at a given time $t>0$),
$t\mapsto d(\cdot,t):=d(\cdot, E(t))$ is left-continuous with respect
to the local uniform convergence.
\end{remark}

\begin{proof} 
By condition  (d) of Definition~\ref{Defsol},  the distributional derivative 
$ \partial_t d$ is a Radon
measure in $\R^N\times (0,T^*)\setminus E$,
so that $d$ is locally a function with bounded
variation
in this (open) domain.
In particular, for a.e. $x\in \R^N$  the map
$s\mapsto d(x,s)$ has a right limit $d^r(x,t)$ at each time $t\in [0, T^*)$ such that $d(x,t)>0$.  Since the functions $d(\cdot,s)$ are also equi-Lipschitz
in space as $s$ varies, we may conclude that  the  right limit is in fact locally uniform in $\{x: d(x,t)>0\}$. 

Since $E$ is closed, for every $t\in [0, T^*]$  we clearly have that all Kuratowski cluster points of $E(s)$ as $s\to t$ are contained in $E(t)$,
equivalently, $d(x,t)\le \liminf_{s\to t} d(x,s)$.
Thus, $d^r\geq d(\cdot, t)$ in $\{x: d(x,t)>0\}$.
\end{proof}

%


\section{Comparison results}\label{sec:comp}

 
In this section we prove the main comparison principle between sub- and supersolutions (see Theorem~\ref{th:compar}). 
In Lemma~\ref{lem:compwulff} below,  we  establish a  first (suboptimal) comparison result between  a supersolution and 
a suitable anisotropic total variation flow (see \cite{BeCaNo02, Moll05}). 
To this aim, we give  an explicit solution to the anisotropic total variation flow  with initial datum $\po$. 

\begin{lemma}\label{lm:explicit}
The pair $(f, \zeta)$ defined by
\begin{equation}\label{eq:soltvf}
f(x,t):=
\begin{cases}
 r(t)+t\frac{N-1}{r(t)} & \textup{ if } \po(x)\le 
r(t):=\sqrt{(N+1)t},\\  
  \po(x) +t\frac{N-1}{\po(x)} & \textup{ otherwise}
\end{cases}
\end{equation}
and 
\begin{equation}\label{eq:soltvfzeta}
\zeta(x,t) :=
\begin{cases}
 \frac{x}{r(t)} & \text{if }\po(x)\le r(t), \\
 \frac{x}{\po(x)} & \text{if }\po(x)\ge r(t), 
\end{cases}
\end{equation}
solve  the following Cauchy problem for the $\p$-total variation flow in $\R^N$:
\begin{equation}\label{eq:tvflow1000}
\begin{cases}
 \partial_t f = \Div \zeta & \text{a.e. in }\R^N\times (0,+\infty),\\
 \zeta\in \pa\p(\nabla f) & \text{a.e. in }\R^N\times (0,+\infty),\\
f(\cdot,0)=\po.
\end{cases}
\end{equation}
Moreover,  given $\lambda>1$, the  pair $(f_\lambda, \zeta_\lambda)$ given by
\[
f_\lambda(x,t) := \lambda f(x,t/\lambda) \qquad \zeta_{\lambda} (x,t):= \zeta(x,t/\lambda)
\]
for $(x,t)\in \R^N\times (0,+\infty)$ solves \eqref{eq:tvflow1000}, with the initial datum $\po$ replaced by $\lambda\po$. 

\end{lemma}
\begin{proof}
Recalling that $\zeta\in \pa\p(\nabla f)$ is equivalent to  $\po(\zeta)\le 1$, $\zeta\cdot\nabla f=\p(\nabla f)$ (see \eqref{subp}), the proof follows by direct verification. The details are left to the reader.
\end{proof}

Next lemma  provides a first  comparison estimate, which is far from being sharp. However, the optimal estimate  can be established  a posteriori as a consequence of our main comparison theorem (see Theorem~\ref{th:compar} below).

\begin{lemma}\label{lem:compwulff}
Let $E$ be a supersolution and $d:=\dist(\cdot, E(\cdot))$ the associated one parameter family of distance functions.  Assume that for some $(\bar x, \bar t)\in \R^N\times [0,+\infty)$ we have  $d(\bar x,\bar t)\geq R>0$.
Then, there exists a constant $\chi_N>0$ 
such that  $d(\bar x,\bar t+s)\ge R-\chi_N\sqrt{s}$ for all $s\in [0, R^2/(16\chi_N^2)]$.
\end{lemma}

\begin{proof}
Observe first that thanks to~Lemma~\ref{lem:rightcont},
since $d(\cdot,\bar t)\ge R/4$
in $\{x:\po(x-\bar x)\le 3R/4\}=W(\bar x,3R/4)$,
there exists a (unknown) time $t^*$ such that 
$d(\cdot, \bar t+s)>\alpha>0$ in $W(\bar x,3R/4)$ for all  $s\in[0, t^*]$ for some positive $\alpha$.
We will compare $d$ with the solution $\delta$
of the $\p$-total variation
flow starting from
\[
\delta(\cdot,0): = R-\frac{4}{3}\po(\cdot -\bar x)\,.
\]
More precisely, setting 
 $\delta(x,s):=R- f_{4/3}(x-\bar x,s)$, where $f_{4/3}(x, t):=4/3f(x, 3t/4)$ and $f$ is given by \eqref{eq:soltvf}, by Lemma~\ref{lm:explicit}
 $\delta$ satisfies
 \begin{equation}\label{eq:delta}
\begin{cases}
 \partial_t \delta = \Div \xi & \text{in } \R^N\times (0,+\infty),\\
 \xi \in \pa\p(\nabla \delta) & \text{a.e. in }\R^N\times (0,+\infty),\\
\end{cases}
\end{equation}
 where $\xi(x,t)= -\zeta(x,3t/4)$, with $\zeta$ defined by \eqref{eq:soltvfzeta}. Note that $\delta$
 is negative outside  $W(\bar x,3R/4)$ for  all positive times.

Let $\Psi(s)$ be a smooth, convex, nonnegative  function,
which vanishes only for $s\le 0$, and consider the function
$w(x,s):=\Psi(\delta(x,s)-d(x,\bar t+s))$. Without loss of generality,
we assume to simplify the notation that $\bar t=0$.
By construction, $w(x,0) \equiv 0$ in $W(\bar x,3R/4)$ and $w(\cdot,s)\equiv 0$
on $\partial W(\bar x,3R/4)$ for $0\le s \le t^*$.

Since $\p(\nabla d)\le 1$~a.e.~and $ \partial_t d $ is a measure wherever it is positive, it follows that $d$
is a function in $BV_{loc}(W(\bar x,3R/4)\times (0,t^*))$ and its
distributional time derivative has the form
\[
 \partial_t d  = \sum_{t\in J}[d(\cdot,t+0)-d(\cdot,t-0)]dx + \partial_t^d d 
\]
where $J$ is the (countable) set of times where $d$ 
jumps and $\partial^d_t d$ is the diffuse (Cantor$+$absolutely continuous)
part of the derivative. It turns out that (Lemma~\ref{lem:rightcont})
$d(\cdot,t+0)-d(\cdot,t-0)\ge 0$ for each $t\in J$. Moreover, since the positive part
of  $\Div z$ is absolutely continuous with respect to the Lebesque measure, 
\eqref{eq:supersol} entails
\[
\partial^d_t d\ge \Div z.
\]

Using the chain rule for $BV$ functions, see~\cite{AmbDM}),
one has
\begin{multline*}
\partial_t w=
 \sum_{t \in J} [\Psi(\delta(\cdot,t)-d(\cdot,t+0) )
-\Psi(\delta(\cdot,t)-d(\cdot,t-0) )]dx 
\\+ 
\Psi'(\delta-d)(\partial_t\delta-\partial^d_t d)
\le \Psi'(\delta-d)(\Div \xi - \Div z).
\end{multline*}
Hence, for a.e.~$t\le t^*$, using the fact that  $\phi$ and  $\Psi$ are convex,  $\Psi'(\delta-d)$
vanishes on $\partial W(\bar x,3R/4)$ and recalling \eqref{eq:delta}, we have
\begin{multline*}
\partial_t \int_{W(\bar x,3R/4)} w dx
\le  \int_{W(\bar x,3R/4)} \Psi'(\delta-d)(\Div \xi - \Div z)
\\= -\int_{W(\bar x,3R/4)} (\xi-z)\cdot(\nabla \delta-\nabla d)\Psi''(\delta-d) \le 0.
\end{multline*}
It follows that $w=\Psi(\delta-d)=0$, 
that is, $d\ge \delta$ a.e.~at all times less than $t^*$. More precisely,
 for $0\leq s\le t^*$ we have
\begin{equation}\label{eq:estimsqrtt}
d(\bar x,\bar t+s) \ge R-f_{4/3}(x-\bar x,s)=R-\frac{4N}{\sqrt{3}}\sqrt{\frac{s}{N+1}}=: R-\chi_N\sqrt{s}.
\end{equation}
It follows from \eqref{eq:estimsqrtt} that
$d(\bar x,\bar t+ s) > 3R/4 $ and, in turn, $d(\cdot ,\bar t + s) >0$ on $ \partial W(\bar x,3R/4)$ for all $s< \min\{t^*,R^2/(16\chi_N^2)\}$.
But then we can restart the argument above to find that  \eqref{eq:estimsqrtt} remains valid for slightly larger times.  Thus, we may conclude that
\eqref{eq:estimsqrtt} holds at least for all 
$0\leq s\leq R^2/(16\chi_N^2)$.  This concludes the proof of the lemma.

\end{proof}

Now we can state the main result of this section, which is a comparison
result between sub- and supersolutions.
\begin{theorem}\label{th:compar}
Let $E$ be a supersolution with initial datum $E^0$ and
$F$ be a subsolution with initial datum $F^0$. Assume that
$\dist(E^0,{F^0}^c)=:\Delta>0$. Then for each $t\ge 0$, $\dist(E(t),F^c(t))\ge \Delta$.
\end{theorem}
\begin{proof}
Let $T^*_E$ and $T^*_F$ be the maximal existence time for $E$ and $F$. For all $t> \min\{T^*_E, T^*_F\}$ we have that either $E$ or $F^c$ is empty. In this case, clearly  the conclusion holds true. 

Now, consider the case $t\le \min\{T^*_E, T^*_F\}$ (and assume without lost of generality that $T^*_E, T^*_F >0$).  Let us fix $0<\eta_1<\eta_1'<\eta_1''<\eta_2''<\eta_2'<\eta_2<\Delta$.  We will show the conclusion of the theorem for a time interval $(0, t^*)$ for a suitable $t^*$ depending only on $\eta_1$, $\eta_1'$, $\eta_1''$, $\eta_2''$, $\eta_2'$, $\eta_2$, and ultimately only on $\Delta$. It is clear then that reiterating the argument yields the conclusion of the theorem for all times. We recall that $d_E(x,t):=\dist (x, E(t))-\dist(x, E^c(t))$ and
 $d_F$ is defined  analogously. We denote by $z_E$ and $z_F$ the  fields appearing in the definition of super- and subsolutions (see Definition \ref{Defsol}), corresponding to $E$ and $F$, respectively.  Define 
$$
S:=\{x\in \R^N:\, \eta_1<d_E(x,0)<\eta_2\}
$$
and note that by Lemma~\ref{lem:compwulff} there exists  $t^*>0$ depending only on $\eta_1$, $\Delta-\eta_2$ such that 
\begin{equation}\label{lemma2}
\begin{array}{l}
d_E(x,t)\geq d_E(x,0)-\chi_N\sqrt{t} \vspace{5pt} \\
d_F(x,t)\leq d_F(x,0)+\chi_N\sqrt{t} 
\end{array}
\qquad\text{for all $x\in \overline S$ and $t\in (0, t^*)$.}
\end{equation}
We now set 
\begin{align*}
& \tilde d_E:=d_E\lor (\eta_1'+\chi_N\sqrt{t})\,,\\
&  \tilde d_F:=(d_F+\Delta)\land (\eta_2'-\chi_N\sqrt{t})\,.
\end{align*}
Clearly, by our assumptions $\tilde d_E(\cdot, 0)\geq \tilde d_F(\cdot, 0)$. We claim that 
\beq\label{parabolicbd}
\tilde d_E\geq \tilde d_F\text{ on }\pa S\times  (0, t^*)\,.
\eeq
Here and in the rest of the proof we may assume without loss of generality that $t^*$ is as small as needed (but still depending only on $\Delta$). 
To this aim, write $\pa S=\Gamma_1\cup\Gamma_2$, where $\Gamma_1:=\{d_E(\cdot, 0)=\eta_1\}$ and $\Gamma_2:=\{d_E(\cdot, 0)=\eta_2\}$.
Since $d_F(\cdot, 0)+\Delta\leq d_E(\cdot, 0)=\eta_1$ on $\Gamma_1$,  we deduce 
$$
\tilde d_F\leq d_F+\Delta\leq \eta_1+\chi_N\sqrt{t}\leq \eta_1'\leq \tilde d_E
$$
on $\Gamma_1\times (0, t^*)$.  Similarly one can show that the inequality  $\tilde d_E\geq \tilde d_F$ holds on $\Gamma_2\times (0, t^*)$. 

Again by \eqref{lemma2} we have
\beq\label{straponzina}
d_E\geq \frac{\eta_1''}{2}>0 \quad\text{in }\{d_E(\cdot, 0)\geq \eta_1''\}\times (0, t^*)
\eeq 
and, observing that $d_F(\cdot, 0)\leq \eta_2''-\Delta$ in  $\{d_E(\cdot, 0)\leq \eta_2''\}$,
\beq\label{straponzina1}
d_F\leq \frac{\eta_2''-\Delta}{2}<0 \quad\text{in }\{d_E(\cdot, 0)\leq \eta_2''\}\times (0, t^*)\,.
\eeq 
In particular
$$
E(t)\subset\subset F(t) \qquad\text{for }t\in (0, t^*)\,.
$$
We now claim that, setting
$$
S'':=\{x\in \R^N:\, \eta''_1<d_E(x,0)<\eta''_2\},
$$ 
we have
\beq\label{claim2}
\tilde d_E=d_E \quad\text{and}\quad \tilde d_F=d_F+\Delta\qquad\text{in }S''\times (0, t^*)\,.
\eeq
Indeed by \eqref{lemma2} we have
$$
d_E(x, t)\geq \eta_1''-\chi_N\sqrt{t}\geq \eta_1+\chi_N\sqrt{t}\quad\text{for }(x,t)\in S''\times (0, t^*)
$$
 and thus $\tilde d_E=d_E$ in $S''\times (0, t^*)$. The proof of the second identity in \eqref{claim2} is analogous.
  
Now we will use quite standard parabolic maximum principles, like in the proof of Lemma \ref{lem:compwulff}.  Notice  that
 $$
 \partial_t\tilde d_E=  \sum_{t\in J}[\tilde d_E(\cdot,t+0)-\tilde d_E(\cdot,t-0)]dx + \partial^d_t \tilde d_E\,,
 $$
 where $J$ is the (countable) set of times where $d_E$ possibly
jumps and $\partial^d_t \tilde d_E$ is the diffuse 
part of the distributional derivative. Using for instance the chain rule proved in \cite{AmbDM}, in $S\times (0, t^*)$ we have that
$$
\partial^d_t \tilde d_E=
\begin{cases}
\frac{\chi_N}{2\sqrt{t}} & \text{a.e. in }\{(x,t):\eta_1'+\chi_N\sqrt{t}>d_E(x)\}\,,\\
\partial^d_t  d_E & \text{$|\partial^d_t  d_E|$-a.e. in }\{(x,t):\eta_1'+\chi_N\sqrt{t}\leq d_E(x)\}\,.
\end{cases}
$$
An analogous formula holds for $\partial^d_t \tilde d_F$. Recalling that 
$(\Div z_E)^+$ and  $(\Div z_F)^- $ belong to  $L^{\infty}(S\times(0,t^*))$  
it follows  that (possibly modifying $t^*$)
\beq\label{zetae0}
\partial^d_t \tilde d_E \geq \Div z_E \quad\text{and}\quad \partial^d_t \tilde d_F \leq \Div z_F
\eeq
in the sense of measures in $S\times (0, t^*)$. Note also that a.e. in $S\times (0, t^*)$
\beq\label{zetae1}
z_E\in \pa \p(\nabla \tilde d_E) \quad\text{and}\quad z_F\in \pa \p(\nabla \tilde d_F)\,.
\eeq
Fix $p>N$ and set $\Psi(s):=(s^+)^p$  and 
$w:=\Psi(\tilde d_F-\tilde d_E)$.  By \eqref{parabolicbd} we have 
\beq\label{asbefore0}
w=0 \qquad\text{on }\pa S\times (0, t^*)\,.
\eeq
Using as before the chain rule for $BV$ functions, recalling \eqref{zetae0} and the fact that the jump parts of $\pa_t\tilde d_E$ and $\pa_t\tilde d_F$ are nonnegative and
nonpositive, respectively,  we have
\beq\label{asbefore}
\partial_t w\leq 
 \Psi'(\tilde d_F-\tilde d_E)(\partial^d_t\tilde d_F-\partial^d_t \tilde d_E)
\le \Psi'(\tilde d_F-\tilde d_E)(\Div z_F - \Div z_E)
\eeq
in $S\times (0, t^*)$. Choose a cut-off function $\eta\in C^{\infty}_c(\R^N)$ such that $0\leq \eta\leq 1$ and $\eta\equiv 1$ on $B_1$. For every $\e>0$ we set $\eta_\e(x):=\eta(\e x)$.
Using \eqref{asbefore0} and \eqref{asbefore}, we have
\begin{align*}
\partial_t \int_{S} w \eta_\e^p dx
&\le  \int_{S}\eta_\e^p \Psi'(\tilde d_F-\tilde d_E)(\Div z_F - \Div z_E)\\
&= -\int_{S} \eta_\e^p\Psi''(\tilde d_F-\tilde d_E)(z_F-z_E)\cdot(\nabla \tilde d_F-\nabla \tilde d_E)\, dx+\\
&\hphantom{\leq}\,\,\, p\int_S\eta_\e^{p-1}\,\Psi'(\tilde d_F-\tilde d_E)\nabla\eta_\e\cdot(z_F-z_E)\, dx\\
 &\le p\int_S\eta_\e^{p-1}\,\Psi'(\tilde d_F-\tilde d_E)\nabla\eta_\e\cdot(z_F-z_E)\, dx,
\end{align*}
where we have also used the inequality $(z_F-z_E)\cdot(\nabla \tilde d_F-\nabla \tilde d_E)\geq 0$, which follows from \eqref{zetae1} and the convexity of $\p$. By H\"older Inequality and using the explicit expression of $\Psi$ and $\Psi'$, we get
$$
\partial_t \int_{S} w \, \eta_\e^p dx\leq Cp^2 \|\nabla \eta_\e\|_{L^p(\R^N)}\left(\int_{S} w\,  \eta_\e^p dx\right)^{1-\frac1p}\,,
$$ 
for some constant $C>0$ depending only on the $L^\infty$-norms of $z_E$ and $z_F$. Since $w=0$ at $t=0$, a simple ODE argument then yields
$$
\int_{S} w \, \eta_\e^p dx\leq \left(Cp\|\nabla \eta_\e\|_{L^p(\R^N)} t\right)^p
$$ 
for all $t\in (0, t^*)$.
Observing that $\|\nabla \eta_\e\|_{L^p(\R^N)}^p=\e^{p-N}\|\nabla \eta\|_{L^p(\R^N)}^p\to 0$ and $\eta_\e\nearrow 1$  as $\e\to 0^+$, we conclude that $w=0$,
and in turn  $\tilde d_E\geq \tilde d_F$   in $S\times (0,t^*)$. 
In particular, by claim \eqref{claim2}, we have shown that 
$d_E\geq d_F+\Delta$   in $S''\times (0,t^*)$. We finally claim that   $\dist (E(t), F^c(t))\geq \Delta$ for $t\in (0, t^*)$. 
To see this, fix $\e\ge 0$, and let
let $x\in \pa E(t)$ and $y\in \pa F(t)$ be such that $\po(x-y)\le \dist (E(t), F^c(t))+ \e$. Note that by \eqref{straponzina} and \eqref{straponzina1}
we have $d_E(x,0)<\eta_1''$ and $d_E(y, 0)>\eta_2''$. Thus there exists $z\in S''\cap [x,y]$, where $[x,y]$ denotes the segment joining $x$ and $y$.  
Since $d_E(\cdot, t)\geq d_F(\cdot, t)+\Delta$   in $S''$, we have
\begin{multline}
\dist (E(t), F^c(t))\ge \po(x-y) - \e = 
\po(x-z) + \po(z-y) - \e \ge
\\
 - d_F(z,t) + d_E(z,t) -\e \ge \Delta - \e.
\end{multline}
The claim follows by the arbitrariness of $\e$,  
and this concludes the proof of the theorem. 
\end{proof}

\section{Existence via minimizing movements}\label{sec:ATW}
In this section we prove an existence result for the crystalline curvature flow, according to Definition \ref{Defsol}. Such a solution is obtained via a variant of the
Almgren-Taylor-Wang minimizing movements scheme (\cite{ATW}) introduced in  \cite{Chambolle, CaCha}.

\subsection{Minimizing movements}

Let $E^0\subset \R^N$ be closed. Fix a time-step $h>0$ and set
$E^0_h=E^0$. We then inductively define $E_h^{k+1}$ (for all $k\in \N$) according to the following  procedure: 
If $E_h^{k}\neq \emptyset$, $\R^N$, then let   $(u_h^{k+1},z_h^{k+1}) :\R^N\to\R\times \R^N$ satisfy
\begin{equation}\label{eq:iterk}
\left\{
\begin{array}{lll}
 -h \, \Div z_h^{k+1} + u_h^{k+1} = \dd_{E_h^k},  \\
 z_h^{k+1} \in \pa\p(\nabla u_h^{k+1}) \quad\text{a.e. in $\R^N$},\\ 
\end{array}
\right.
\end{equation}
and  set  $E_h^{k+1}:=\{x:u_h^{k+1}\le 0\}$. If either $E_h^{k}=\emptyset$ or $E_h^{k}=\R^N$, then set $E_h^{k+1}:=E_h^{k}$. We denote by $T^*_h$ the first discrete time $hk$ such that $E_h^k=\emptyset$, if such a time exists; otherwise we set  $T^*_h=+\infty$.

In proposition \ref{prop:ATW} below we will show that this construction  is well defined, since  problem \eqref{eq:iterk} admits a unique solution $u_h^{k+1}$ that is Lipschitz continuous. In particular, $E_h^{k+1}$ is a closed set for all $k$. 

Before stating  the main facts about the differential problem \eqref{eq:iterk}, we recall
that given $z\in L^{\infty}(\R^N; \R^N)$ with $\Div z\in L^2_{loc}(\R^N)$ and  $w\in BV_{loc}(\R^N)\cap L^2_{loc}(\R^N)$,
 $z\cdot Dw$ denotes the Radon measure associated with the linear functional 
$$
L\varphi:=-\int_{\R^N}w\,\varphi \Div z\, dx- \int_{\R^N}w\,z\cdot \nabla \varphi \, dx \qquad \text{ for all } \varphi\in C^{\infty}_c(\R^N),
$$
see~\cite{Anz:83}.


\begin{proposition}\label{prop:ATW}
Let $g\in L^2_{loc}(\R^N)$.
There exists a field $z\in L^\infty(\R^N;W(0,1))$ and a  
unique function
$u \in BV_{loc}(\R^N)\cap L^2_{loc}(\R^N)$
such that the pair $(u,z)$ satisfies 
\begin{equation}\label{eq:iterk2}
\left\{
\begin{array}{ll}
 -h \, \Div z  + u = g \qquad &\text{ in } \mathcal D'(\R^N),  \\
 \po(z)\le 1 \quad & \text{ a.e. in } \R^N,\\
z\cdot Du = \p(Du) \qquad &\text{ in  the sense of measures}.
\end{array}
\right.
\end{equation}
Moreover, for any $R>0$ and $v\in BV(B_R)$ with $\spt (u-v)\Subset B_R$,
\[
\p(Du)(B_R) +\frac{1}{2h} \int_{B_R}{(u-g)^2}\, dx \le
\p(Dv)(B_R) +\frac{1}{2h} \int_{B_R}{(v-g)^2}\, dx,
\]
and for every $s\in \R$ the set $E_s:=\{x\in \R^N: \, u(x)\le s\}$ solves the minimization problem
\[
\min_{F\Delta E_s\Subset B_R}  P_\p(F;B_R) + \frac{1}{h}  \int_{F\cap B_R} (g(x) -s) \, dx.
\]

If $g_1\leq g_2$ and if $u_1$, $u_2$ are the corresponding solutions to \eqref{eq:iterk2} (with $g$ replaced by $g_1$ and $g_2$, respectively), then $u_1\leq u_2$. 

Finally if in addition $g$ is Lipschitz with $\p(\nabla g)\le 1$,
then  the unique solution $u$ of \eqref{eq:iterk2}
is also Lipschitz 
and satisfies  $\p(\nabla u)\le 1$ a.e. in $\R^N$. As a consequence, \eqref{eq:iterk2} is equivalent to 
\begin{equation}\label{eq:iterk3}
\left\{
\begin{array}{lll}
 -h \, \Div z  + u = g, \\
 z  \in \pa\p(\nabla u) \quad\text{a.e. in $\R^N$}\\ 
\end{array}
\right.
\end{equation}
\end{proposition}

\begin{proof} See \cite[Theorem~2]{CaCha}, \cite[Theorem~3.3]{AlChNo}. 
\end{proof}
\begin{remark}[Consistency with the ATW scheme] When $\partial E^0$ is bounded, the minimality property of the level sets stated above shows, in particular, that
the sets $E_h^{k}$ are constructed  according to the Almgren-Taylor-Wang scheme~\cite{ATW}.
\end{remark}
Since by the previous proposition  $\p(\nabla u_h^{k+1})\le 1$~a.e.~in $\R^N$,
one deduces, in particular, that
\begin{equation}\begin{array}{ll}\label{eq:ineqd}
u_h^{k+1} \le d_{E_h^{k+1}} & \textup{ in } \{x:\dist(x, E_h^{k+1})>0\}\,,\\
u_h^{k+1} \ge d_{E_h^{k+1}} & \textup{ in } \{x:\dist(x, E_h^{k+1})<0\}\,.
\end{array}
\end{equation}

We are now in  a position to define the time discrete evolutions. Precisely, we set  
\begin{equation}\label{discretevol}
\begin{array}{l}
E_h:=\{(x,t): x\in E_h^{[t/h]}\},\vspace{2pt}\\
 E_h(t):=E_h^{[t/h]}=\{x: (x,t)\in E_h\},\vspace{2pt}\\
d_h(x,t):=\dd_{E_h(t)}(x),\vspace{2pt}\\
 u_h(x,t):=u_h^{[t/h]}(x),\vspace{2pt} \\
z_h(x,t):=z_h^{[t/h]}(x),
\end{array}
\end{equation}
where $[\cdot]$ stands for the integer part of its argument.

\begin{remark}[Discrete comparison principle]\label{rm:dcp}
The last part of  Proposition~\ref{prop:ATW} clearly implies that the  scheme is monotone, that is, the discrete evolutions satisfy the comparison principle. More precisely, if $E^0\subseteq F^0$ are closed sets and if we denote by $E_h$ and $F_h$ the  discrete evolutions with initial datum $E^0$ and $F^0$, respectively, then $E_h\subseteq F_h$. 
\end{remark}

\subsection{Comparison with the Wulff shape} In this subsection, we exploit Remark~\ref{rm:dcp} to compare the discrete evolutions  \eqref{discretevol} with the minimizing movements  of the Wulff shape and derive an estimate, which will be useful in the convergence analysis. 
The evolution starting from a Wulff shape $W(0,R)$ is explicitly
known. Indeed, from~\cite[Appendix B,  Eq.~(39)]{CaCha},
the solution of \eqref{eq:iterk2}, with $g$ replaced by $\dd_{W(0,R)}=\po-R$,
is given by $\po_h-R$, where 
\begin{equation}\label{eq:explicitpoh}
\po_h(x): =\begin{cases}
\sqrt{h}\frac{2N}{\sqrt{N+1}} & \textup{ if } \po(x)\le \sqrt{h(N+1)},\\
\po(x)+h\frac{N-1}{\po(x)} & \textup{ else.}
\end{cases}
\end{equation}
It follows that if $E^0=W(0,R)$, one has $E_h(t)=W(0,r_h^R(t))$ for
a function
$r_h^R$ that satisfies
$$
r_h^R(h)=\frac{R+\sqrt{R^2-4h(N-1)}}{2} 
$$
if $h\leq R^2/(4(N+1))$.
In particular,
$$
r_h^R(h)\geq \sqrt{R^2-4h(N-1)}
$$
for the same  $h$'s. 
By iteration,  we have 
$r_h^R(t)\geq \sqrt{R^2-4t(N-1)}\geq \frac{R}{\sqrt 2}$
for $0\leq t\leq R^2/(8(N-1))$ and $h\leq R^2/(8(N+1))$. Since $r_h^R(t)=R$ for $t\in [0, h)$, we infer 
\begin{equation}\label{errehstima}
r_h^R(t)\geq \sqrt{R^2-4t(N-1)}
\end{equation}
for $0\leq t\leq R^2/(8(N+1))$ and for all $h$.

Now we return to  the motion from an arbitrary set $E^0$.
If for some $(x,t)\in\R^N\times [0,T_h^*)$ we have $d_h(x,t)>R$, then $W(x,R)\cap E_h(t)=\emptyset$.
%
Hence, by the comparison principle stated in Remark~\ref{rm:dcp} and by  \eqref{errehstima} 
we have
$$
d_h(x,s)\geq \sqrt{R^2-4(N-1)(s-t+h)}
$$
for $t<s$ and $s+h-t <R^2/(8(N+1))$. 


By letting $R\nearrow d_h(x,t)$ we obtain
\beq\label{straponzina21}
d_h(x, s)\geq \sqrt{d_h^2(x,t)-4(N-1)(s-t+h)}
\eeq
for $t<s$ and $s+h-t <d_h^2(x,t)/(8(N+1))$. 



\subsection{Convergence of the scheme} Up to a subsequence we have

\[
E_{h_l}\stackrel{\mathcal K}{\longrightarrow} E\qquad\text{and}\qquad 
{(\mathring{E}_{h_l})}^c\stackrel{\mathcal K}{\longrightarrow} A^c
\]
for a suitable closed sets $E$ and a suitable open set $A\subset  E$. 
Define $E(t)$ and $A(t)$ as in \eqref{discretevol}. 

Observe that if $E(t)=\emptyset$ for some $t\ge 0$,
then~\eqref{straponzina21} implies that $E(s)=\emptyset$ for all $s\ge t$
so that we can define, as in Definition~\ref{Defsol}, the extinction
time $T^*$ of $E$, and similarly the extinction  time ${T'}^*$ of
 $A^c$. Notice that at least one between  $T^*$ and ${T'}^*$ is $+\infty$. 
Possibly extracting a further subsequence, we have the following result:

\begin{proposition}\label{prop:E}
There exists a countable set $\mathcal N\subset (0, +\infty)$ such that
${d_{h_l}}(\cdot, t)^+\to \dist(\cdot, E(t))$   and  $d_{h_l}(\cdot, t)^-\to \dist(\cdot,A^c)$ locally uniformly  for all $t\in (0, +\infty) \setminus \mathcal{N}$.

Moreover, for every $x\in \R^N$ the functions $\dist(x,E(\cdot))$ and  $\dist(x,A^c)$ are left continuous
and are right lower semicontinuous.
Equivalently,  the functions $ E(\cdot)$ and $ A^c$  are left continuous
and are right upper semicontinuous with respect to the Kuratowski convergence.  
Finally, $E(0)= E^0$ and  $A(0)=\mathring{E^0}$.
\end{proposition}



\begin{proof}
By the Ascoli-Arzel\`a Theorem and a standard diagonal argument, 
we may extract a further (not relabeled) subsequence such that 
 $d_{h_l}(\cdot, t)\to  d(\cdot, t)$  locally uniformly for all $t\in \Q\cap (0,+\infty)$, where $d(\cdot, t)$ is either a Lipschitz function or infinite everywhere. In the latter case, either $d(\cdot, t)\equiv +\infty$ or    $d(\cdot, t)\equiv -\infty$. 

We  observe that for all $t\in (0, T^*)\cap \Q$ we have $d(\cdot, t)<+\infty$. To see this we argue by contradiction assuming that for every $x\in \R^N$ and for every $M>0$ we have
$d_{h_l}(x, t)>M$ for all $l$ large enough. We may now apply \eqref{straponzina21} to deduce that there exists a right interval $(t, t')$ independent of $l$ such that $d_{h_l}(x, s)>\frac{M}2$ for $l$ large enough and for all $s\in (t,t')$; that is, $d_{h_l}(\cdot, s)\to+\infty$ for all 
$s\in (t, t')$. This in turn would imply $E(s)=\emptyset$ for all $s\in (t, t')$, which is impossible since $t<T^*$.
A similar argument shows that for all $t\in (0,{T'}^*)$ we have $d(\cdot, t)>-\infty$.

Let $x\in \R^N$ and $t>0$ be such that $\limsup_l d_{h_l}(x,t)=:R>0$. Then, given $0<R'<R$, we have $d_{h_l}(x,t)\ge R'$ for infinitely many $l$. By \eqref{straponzina21}, for $t<s< t+C {R'}^2$ we deduce $\limsup_l d_{h_l}(x,s)>\sqrt{{R'}^2 - 4(N-1) (s-t)}$. In particular if $s\in \Q\cap (t, t+CR^2)$, then  it follows $d(x,s)\ge \sqrt{R^2-4(N-1)(s-t)}$.  If, in addition, also $t\in\Q$, then we have 
$d(x,s)\ge \sqrt{d(x,t)^2-4(N-1)(s-t)}$.
Now let $x\in\R^N$ and $t\geq 0$, and assume $R:=\limsup_{s\in\Q, s\searrow t} d(x,t) >0$. Consider a sequence of rational numbers  $s_k\searrow t$ such that $\lim_k d(x,s_k) = R$. For $s>t$ rational and  close enough to $t$, if $k$ is sufficiently large, then $s_k<s$ and
$d(x,s) \ge \sqrt{d(x,s_k)^2-4(N-1)(s-s_k)}$. Sending $k$ to infinity it follows $d(x,s)\ge \sqrt{R^2-4(N-1)(s-t)}$ so that 
$\liminf_{s\in\Q, s\searrow t} d(x,t) \ge R$. Hence $d(x,\cdot)$ has a right limit at $t$. The same conclusion holds if $\liminf_{s\in\Q, s\searrow t} d(x,t) <0$, with the same proof. 
We deduce that the $d$ admits a right limit (locally uniformly in space) at any $t\ge 0$. A similar argument shows that
$d$ also admits a left limit at any $t > 0$. Moreover, arguing similarly and 
using~\eqref{straponzina21} again, we can show  that
\beq\label{straponzina22}
\begin{array}{rcccl}
d(x,t+0)^\pm & \!\ge\! &
\displaystyle \limsup_{l\to\infty, s\to t} d_{h_l}(x,s)^\pm 
\\[2mm]
& \!\ge\! 
& \displaystyle \liminf_{l\to\infty,s\to t} d_{h_l}(x,s)^\pm
 & \!\ge\! & d(x,t-0)^\pm.
\end{array}
\eeq

Let $\mathcal N$ be the set of all times $t$ such that the left and
right limits of $d$ differ at $(x,t)$, for some $x\in\R^N$
(we also assume $0\in\mathcal{N}$).
Notice that $\mathcal N$ is countable, since it can be written as the union over $k\in \N$ and $x\in \Q^N$ of the times such that the gap between the right and left limit of $d(x,\cdot)$ is larger than $1/k$  (which for fixed $k$ and $x$  cannot have cluster points). We denote by $d(x, t)$ the common value of the right and left limits of $d(x, \cdot)$ at $t\not\in \mathcal N$.  

By \eqref{straponzina22} we immediately have that $\lim_{l\to\infty}d_{h_l}(\cdot , t)= d(\cdot,t)$ for all $t\not\in \mathcal{N}$.  
We now show that for $t\not\in \mathcal{N}$,
we have  $d(\cdot,t)^+=\dist(\cdot ,E(t))$.
This is equivalent to showing that $E(t)$ coincides with the
Kuratowski limit $K$ of $E_{h_l}(t)$, since $d(\cdot,t)^+=\dist(\cdot,K)$.
Clearly, $K\subseteq E(t)$.
Conversely, if $x\not\in K$, 
then $d(x,t)^+=:R>0$. Since $d$ is continuous at $t$, we may find $\e$ so small that  
$\lim_{l\to\infty}d_{h_l}(x,t-\e)\geq d(x,t-\e)>R/2$ and in turn,   by  \eqref{straponzina21},  $W(x,R/4)\times [t-\e,t+\e]\cap E_{h_l}=\emptyset$ for $l$
large enough. Thus  $x\not\in E(t)$, showing that 
  $E(t)=K$ and $d(x,t)^+=\dist(x ,E(t))$. A similar argument yields  that $d(x,t)^-=\dist(x ,A^c)$.
 
 Always by \eqref{straponzina21}, one can easily prove  that $E(0)\subseteq E^0$. Since $E_{h_l}(0)=E^0$ for all $l$, we infer the equality $E(0)= E^0$. Symmetrically, one can show that $A(0)=\mathring{E^0}$. 
 
Finally, we prove the continuity properties of $E(t)$. The right upper semicontinuity with respect to the Kuratowski convergence is a consequence of the fact that $E$ is closed. Let us prove now the left continuity. To this aim, denote by $\hat K$  the Kuratowski limit of $E(s)$ as $s\nearrow t$. Clearly $\hat K\subseteq E(t)$. 
 Let now $x\not\in \hat K$.  Then $\lim_{s\nearrow t}\dist(x, E(s))=\dist(x, \hat K)=:R>0$. Arguing exactly as before we may choose $\e$ so small  that 
 $\liminf_{l}\dist(x, E_{h_l}(t-\e))\geq \dist(x, E(t-\e))>R/2$ and  $W(x,R/4)\times [t-\e,t+\e]\cap E_{h_l}=\emptyset$ for all $l$
large enough, so that $x\not\in E(t)$. Hence $\hat K=E(t)$. This establishes the Kuratowski left-continuity of $E(\cdot)$ and concludes the proof of the proposition.
\end{proof}

\begin{theorem}
$E$ is a supersolution in the sense of Definition~\ref{Defsol}
with initial datum $E^0$,
while $A$ is a subsolution with initial datum ${E}^0$.
\end{theorem}
\begin{proof}
Points \textit{(a), (b)} and \textit{(c)} of Definition~\ref{Defsol} follow from Proposition~\ref{prop:E}.
It remains to show \textit{(d)}.
Possibly extracting a further subsequence and setting $z_{h_l}(\cdot, t):=0$ for $t>T^*_{h_l}$ if $T^*_{h_l}<T^*$, we may assume
that $z_{h_l}$ converges  weakly-$*$ in $L^\infty(\R^N\times (0,T^*))$  to some vector-field $z$
satisfying $\po(z)\le 1$ almost everywhere.
Recall that by~\eqref{eq:ineqd} we have $u_h^{k+1} \le \dd_{E_h^{k+1}}$,
whenever $\dd_{E_h^{k+1}}\ge 0$. In turn,  it follows from~\eqref{eq:iterk} that 
\begin{equation}\label{eq:ineqdisc}
 \Div z_h^{k+1} \le \frac{\dd_{E_h^{k+1}} - \dd_{E_h^k} }{h} \qquad \text{ a.e. on } \{\dd_{E_h^{k+1}}\ge 0\}.
\end{equation}
Consider
a nonnegative test function
$\eta\in C_c^\infty((\R^N\times (0,T^*))\setminus E)$.
If $l$ is large enough, then the distance of the support of $\eta$
from $E_{h_l}$ is bounded away from zero. In particular,  $d_{h_l}$ is finite
and positive on $\mathrm{Supp\,}\eta$. We deduce from~\eqref{eq:ineqdisc} that
\begin{multline*}
\int\int \eta(x,t)\left(\frac{d_{h_l}(x,t+{h_l})-d_{h_l}(x,t)}{h_l}-\Div z_{h_l}(x,t+{h_l})\right)
dt dx 
\\ 
=-\int   \int \left(\frac{\eta(x,t)-\eta(x,t-h_l)}{h_l}d_{h_l}(x,t)- z_{h_l}(x,t+h_l)\cdot\nabla \eta(x,t)\right)
\,dt dx
\ge 0.
\end{multline*}
Passing to the limit $l\to\infty$ we obtain \eqref{eq:supersol}.

Next, we establish an upper bound for $\Div z_{h_l}$ away from $E_{h_l}$. To this aim
 observe that 
\[
\dd_{E^k_h}  = \min_{y\in E_h^k} \po(\cdot - y)
\]
so that, by  \eqref{eq:iterk} and the comparison principle stated at the end of Proposition~\ref{prop:ATW}, 
\[
u_h^{k+1} \le \min_{y\in E_h^k} \po_h(\cdot - y)
\]
where $\po_h$ is given in~\eqref{eq:explicitpoh}. Thus,  if $\dd_{E^k_h}(x)\geq R>0$,
then
\[
u_h^{k+1}(x)\le \min_{y\in E_h^k} \po(x - y)+h\frac{N-1}{R}
= \dd_{E^k_h}(x) +h\frac{N-1}{R},
\]
provided $h\le R^2/(N+1)$.
As a consequence of~\eqref{eq:iterk}, we obtain
\begin{equation}\label{numero}
\Div z_h^{k+1}\le  \frac{N-1}{R} \qquad\text{a.e. in }\{x:\, \dd_{E^k_h}(x)\geq R\}.
\end{equation}
It is then easy to deduce from the convergence properties
of $E_{h_l}$ and $d_{h_l}$ that 
\[
\Div z \le \frac{N-1}{R}\qquad\text{in }\{(x,t)\in\R^N\times (0, T^*):\, d(x,t)>R\}
\]
in the sense of distributions. It follows that $\Div z$
is a Radon measure in $\R^N\times (0,T^*)\setminus E$, and 
$(\Div z)^+\in L^\infty(\{(x,t)\in\R^N\times (0,T^*):\, d(x,t)\geq\delta\})$ for every $\delta>0$.

We now provide  a lower ($h$-dependent) bound for $\Div z_{h_l}$.  To this aim, note that if $\dd_{E^k_h}(x)=:R>0$, then 
$\dd_{E^k_h}\geq R-\po(\cdot-x)$. Thus, by comparison as before, 
$$
u_h^{k+1}(x)\geq R-\po_h(0)=R-\sqrt{h}\frac{2N}{\sqrt{N+1}}\,.
$$
In turn, by \eqref{eq:iterk}, we deduce
$$
\Div z_{h}^{k+1}\geq -\frac{1}{\sqrt h} \frac{2N}{\sqrt{N+1}} \qquad\text{a.e. in }\{x:\, \dd_{E^k_h}(x)>0\}.
$$
Combining the above inequality with \eqref{numero} and  using \eqref{eq:iterk} again,  we deduce that for all $t\in (0, T^*)\setminus\mathcal{N}$ (where recall that $\mathcal{N}$ is introduced in Proposition \ref{prop:E}) and any $\delta>0$
$$
\|u_{h_l}(\cdot, t)-d_{h_l}(\cdot, t-h_l)\|_{L^{\infty}(\{x:d_{h_l}(x,t-h_l)\geq\delta\})}\leq \sqrt{h}\frac{2N}{\sqrt{N+1}},
$$
provided that $l$ is large enough.  In particular, recalling  the convergence properties
of $E_{h_l}$ and $d_{h_l}$ (see also \eqref{straponzina22}), we deduce that 
\begin{equation}\label{elleuno}
u_{h_l}\to d\qquad\text{a.e. in }\R^N\times (0,T^*)\setminus E,
\end{equation}
with the sequence $\{u_{h_l}\}$ locally (in space and time) uniformly bounded. 

Consider now, as before,  a nonnegative test function
$\eta\in C_c^\infty((\R^N\times (0,T^*))\setminus E)$. 
Then, recalling \eqref{elleuno}, we have by lower semicontinuity
\[
\int\int \p(\nabla d)\eta\, dxdt\le \liminf_{l}\int\int \p(\nabla u_{h_l})\eta\, dxdt
= \liminf_l \int\int (z_{h_l}\cdot\nabla u_{h_l})\eta\, dxdt.
\]
On the other hand,
\[
\int\int (z_{h_l}\cdot\nabla u_{h_l}) \eta\, dxdt = 
\int\int (z_{h_l}\cdot\nabla d) \eta\, dxdt + 
\int\int z_{h_l}\cdot\nabla (u_{h_l}-d) \eta\, dxdt,
\]
with
\[
\int\int (z_{h_l}\cdot\nabla d) \eta\, dxdt
 \stackrel{l\to\infty}{\longrightarrow}
\int\int (z\cdot\nabla d) \eta\, dxdt.
\]
Hence, we obtain
\begin{equation}\label{eq:verygood}
\int\int \p(\nabla d)\eta\, dxdt\le \int\int (z\cdot\nabla d) \eta\, dxdt,
\end{equation}
provided we show that 
\begin{equation}\label{eq:good}
\lim_l \int\int z_{h_l}\cdot\nabla (u_{h_l}-d) \eta\, dxdt = 0.
\end{equation}
For each $t$, set
\[m_l(t):=\min_{x\in\spt \eta(\cdot,t)} \bigl(u_{h_l}(x,t)-d(x,t)\bigr),\quad
M_l(t):=\max_{x\in \spt \eta(\cdot,t)} \bigl(u_{h_l}(x,t)-d(x,t)\bigr).\]
Recall that these quantities are uniformly
bounded and converge to $0$ at all $t\not\in \mathcal{N}$.
Then, we can write
\begin{multline}\label{yeswecan}
\int\int z_{h_l}\cdot\nabla (u_{h_l}-d) \eta\, dxdt 
=
\int\int z_{h_l}\cdot\nabla (u_{h_l}-d-m_l) \eta\, dxdt 
\\ = -\int\int (u_{h_l}-d-m_l) (z_{h_l}\cdot \nabla \eta + \eta\Div z_{h_l})\, dxdt.
\end{multline}
For $l$ large enough, since the support of $\eta$ is at
positive distance from $E$ there exists $\delta>0$ such that
$d_{h_l}\geq \delta$ everywhere on this support, so that $\Div z_{h_l}\le(N-1)/\delta$.
It follows that
\[
 -\int\int (u_{h_l}-d-m_l)  \eta\Div z_{h_l}\, dxdt
\ge 
-\frac{N-1}{\delta}\int\int (u_{h_l}-d-m_l) \eta\, dxdt\stackrel{l\to\infty}{\longrightarrow}0,
\]
thanks also to \eqref{elleuno}.
Recalling \eqref{yeswecan}, we can conclude that
\[
\liminf_l 
\int\int z_{h_l}\cdot\nabla (u_{h_l}-d) \eta\, dxdt 
\ge 0.
\]
In the same way, writing now
\[
\int\int z_{h_l}\cdot\nabla (u_{h_l}-d) \eta\, dxdt 
=
\int\int z_{h_l}\cdot\nabla (u_{h_l}-d -M_l) \eta\, dxdt 
\]
and using $u_{h_l}-d-M_l\le 0$ a.e.~on $\spt\eta$, one can show that
\[
\limsup_l 
\int\int z_{h_l}\cdot\nabla (u_{h_l}-d) \eta\, dxdt 
\le 0
\]
so that \eqref{eq:good} follows.
In turn, \eqref{eq:verygood} holds, that is, $\p(\nabla d)\leq z\cdot \nabla d$ a.e. in $\R^N\times (0,T^*)\setminus E$.
On the other hand, recalling that $\po(z)\leq 1$ a.e. in $\R^N\times (0, T^*)$, we have
$$
z\cdot\nabla d 
\leq \p(\nabla d)
$$
a.e. in $\R^N\times (0, T^*)$. We conclude that  $\p(\nabla d)= z\cdot \nabla d$ and, in turn, $z\in \partial\p(\nabla d)$ a.e. in $\R^N\times (0,T^*)\setminus E$. This concludes the proof that $E$ is a supersolution.
 The proof that $A$ is a subsolution is identical.
\end{proof}
\begin{corollary}
Let $u^0$ be a bounded, uniformly continuous in $\R^N$. Then for
all $s\in \R$ but a countable number, the minimizing movement scheme
starting from  $E^0_s=\{u^0 \le s\}$ converges to the unique solution
of the curvature flow
in the sense of Definition~\ref{Defsol}, with initial datum $E^0_s$.
\end{corollary}
\begin{proof} The arguments are standard and
rely on the comparison theorem~\ref{th:compar}. The bad (countable set)
is the set of levels for which ``fattening'' occurs, that is,
$|E\setminus A|>0$. Observe that from Theorem~\ref{th:compar},
one easily shows the existence of a unique level-set solution $u(x,t)$
starting from $u^0$, which shares the same spatial modulus of continuity
and is also uniformly continuous in time (see for instance \cite[Subsection~6.3]{CMP3}).
\end{proof}

\section{Conclusion and perspectives}
In this note we have shown 
the existence and uniqueness of a mean curvature flow
(namely, the ``natural'' flow by mean curvature along the Cahn-Hoffmann
vector field) with a technique which does not require any type of
regularity on the surface tension, and thus have provided the first
sound definition of a crystalline curvature flow in any dimension.
It does not require that the initial surface is bounded and applies,
in particular, also to the case of graphs.
The uniqueness result is based on a very standard parabolic comparison
principle. The general approach,
based on the fact that the level sets of the distance
functions have nonincreasing curvatures as the distance increases
(as was exploited as early as in~\cite{Soner93} in the viscosity setting),
can quite probably be used in more general situations, and even
maybe for motions which are not necessarily variational.
However, it should need substantial adaption. For instance,
if replacing
the mobility $m=\po$ in our approach by other (convex) functions
is in principle easy (it is enough to consider, for the distance
functions, the $m$-distance function instead of the $\po$-distance),
in the nonsmooth case it yields difficulties which still require further
investigation. Indeed, if $m$ is smooth and $\p$ is not, then it will not 
be true anymore that the level sets of the distance function 
have globally bounded curvature as the distance increases, so that
Definition~\ref{Defsol} needs to be changed. It is not
yet clear what assumption
on $(\Div z)^\pm$ is then useful in order to be able to derive both existence and
uniqueness. This is a subject for future study.

\bibliography{ccf}

\end{document}